\documentclass[12pt]{article}
\usepackage{amsmath,amsfonts,amssymb,amsthm}
\newtheorem{theo}{Theorem}
\newtheorem{lem}{Lemma}

\newcommand\la{\lambda}

\title{Narayana polynomials and\\  
Hall-Littlewood symmetric functions}
\author{Michel Lassalle\\
\small Centre National de la Recherche Scientifique\\[-0.8ex]
\small Institut Gaspard-Monge, Universit\'e de Marne-la-Vall\'ee\\[-0.8ex]
\small 77454 Marne-la-Vall\'ee Cedex, France\\[-0.8ex]
\small \texttt{lassalle@univ-mlv.fr}\\[-0.8ex]
\small \texttt{http://igm.univ-mlv.fr/{\textasciitilde}lassalle}
}
\date{}
\begin{document}
\maketitle

\begin{abstract}
We show that Narayana polynomials are a specialization of row Hall-Littlewood symmetric functions. Using $\la$-ring calculus, we generalize to Narayana polynomials the formulas of Koshy and Jonah for Catalan numbers.\\


\end{abstract}

\section{Introduction}

There are many possible $q$-generalizations of Catalan numbers~\cite[exercise 34]{S}. In this paper we shall deal with the following one, perhaps less popular than others.

Let $q$ be an indeterminate. For any nonnegative integer $n$, we denote by $C_n(q)$ the polynomial in $q$ defined by $C_0(q)=1$ and the recurrence formulas
\begin{align*}
C_n(q)&=(1-q)C_{n-1}(q)+q\,\sum_{i=0}^{n-1}C_i(q)C_{n-i-1}(q)\\
&=(1+q)C_{n-1}(q)+q\,\sum_{i=1}^{n-2}C_i(q)C_{n-i-1}(q),
\end{align*}
the second relation valid for $n\ge3$. Clearly we have $C_n(0)=1$, $C_n(1)=C_n$, the ordinary Catalan number, and $C_n(2)=s_n$, the small Schr\"{o}der number~\cite{G}. The first values of $C_n(q)$ are given by
\begin{align*}
C_1(q)=1,\quad C_2(q)=q+1,&\quad C_3(q)=q^2+3q+1,\\ 
C_4(q)=q^3+6q^2+6q+1, &\quad C_5(q)=q^4+10q^3+20q^2+10q+1.
\end{align*}

Experts will at once recognize the symmetric distribution of the Narayana numbers $N(n,k)$ defined by
\[ N(n,k)=\frac{1}{n} \binom{n}{k-1} \binom{n}{k}.\]
Using known results about generating functions~\cite{S,Z}, we have
\[C_n(q)=\sum_{k=1}^{n} N(n,k) q^{k-1},\]
the Narayana polynomial. It is sometimes more simple to consider the ``large'' Narayana polynomial, defined by $\mathcal{C}_0(q)=1$ and $\mathcal{C}_n(q)=q\,C_n(q)$ for $n\ge1$. Then we have $\mathcal{C}_n(2)=R_n$, the large Schr\"{o}der number~\cite{G}. 

There is a rich combinatorial litterature on this subject. Here we shall only refer to~\cite{MS,Su} and references therein. The purpose of this paper is to show that $C_n(q)$ has a deep connection with the theory of symmetric functions. 

Let $P_k(x;q)$ denote the one-row Hall-Littlewood symmetric function associated with $q$~\cite[Chapter 3]{Ma}. It is known that $P_k(q)$ interpolates between the power-sum $p_k$ and the complete symmetric function $h_k$, namely
\[P_k(0)=h_k,\qquad P_k(1)=p_k.\]
We denote by $P_k(1^{m};q)$ the value of $P_k(x;q)$ taken at the $m$-vector $x=(1,\ldots,1)$. In Section 3, we prove
\begin{equation*}
C_n(1-q)=\frac{1}{n+1}P_n(1^{n+1};q).
\end{equation*}
This result might be obtained by using classical properties of symmetric functions~\cite{Ma}. However it is more interesting to prove it in the framework of $\lambda$-rings.

The powerful formalism of $\lambda$-rings not only allows a compact presentation. It gives a quick access to deep properties, which would be much more difficult to get otherwise. We shall illustrate this efficiency through several examples. 

A recursive formula for Catalan numbers is given by Koshy as follows~\cite[p.322]{K}
\[C_n=\sum_{k=1}^n(-1)^{k-1}\binom{n-k+1}{k}C_{n-k}.\]
In Section 4, we generalize this relation to Narayana polynomials as
\[C_n(q)=(1-q)^{n-1}+q\sum_{k=1}^{n-1}C_{n-k}(q)
\sum_{m=0}^{k-1} (-1)^m\binom{k-1}{m}\binom{n-m}{k} (1-q)^{k-m-1}.\]
Similarly in Section 5, we extend Jonah's formula for Catalan numbers~\cite[p.325]{K}
\[\sum_{k=0}^r\binom{n-2k}{r-k}C_{k}=\binom{n+1}{r}\]
in two ways. The first one is given by
\begin{multline*}
C_r(q)+q\sum_{k=1}^{r-1} C_{r-k}(q) 
\sum_{m=0}^{k-1}(q-1)^{m}\binom{k-1}{m}\binom{n-2r+2k-m}{k}
\\
=\sum_{m=0}^{r-1}(q-1)^{m}\binom{r-1}{m}\binom{n-m}{r-1}.
\end{multline*}
The second generalization is given for large Narayana polynomials and may be written as
\[\sum_{k=0}^r \mathcal{C}_k(q) \sum_{m=0}^{r-k}(1-q)^m \binom{n-2k-m}{r-k-m}\binom{k+m}{m}=\binom{n+1}{r}.\]
Finally in Section 6, we determine the transition matrix between the large Narayana polynomials corresponding to two variables $q$ and $q^{\prime}$, namely 
\[
(n+1)\mathcal{C}_n(q^{\prime})=\sum_{k=0}^n \mathcal{C}_{n-k}(q)\sum_{i+j=0}^{k}
(1-q)^i(q^{\prime}-1)^j \binom{n-k+i}{i} \binom{n+1}{j} \binom{2k-i-j-1}{k-i-j}.
\]
We emphasize that these four examples are easy consequences of \textit{the same} elementary $\la$-ring identity. It would be very interesting to obtain a combinatorial interpretation of these algebraic results. 

It is a pleasure to thank Christian Krattenthaler, Alain Lascoux and Doron Zeilberger for useful remarks. 

\section{Some identities}

We consider the generating function 
\[\mathbf{C}_q(u)=1+\sum_{r\ge 1}C_r(q)u^r.\] 
By definition it satisfies the relation
\begin{equation}\label{equnun}
(\mathbf{C}_q(u)-1)/u=(1-q)\mathbf{C}_q(u)+q\,\mathbf{C}_q(u)^2,
\end{equation}
namely
\begin{equation*}
qu\,\mathbf{C}_q(u)^2+(u(1-q)-1)\mathbf{C}_q(u)+1=0.
\end{equation*}
The only regular solution at $u=0$ is 
\begin{equation}\label{equnde}
\mathbf{C}_q(u)=\frac{1-(1-q)u-\sqrt{(1-q)^2u^2-2(1+q)u+1}}{2qu}.
\end{equation}
It is proved in~\cite{Z} that this solution is given by
\[\mathbf{C}_q(u)=\sum_{n\ge0, k\ge 0}N(n,k+1) u^n q^k.\] 
Another proof may be found in~\cite[exercise 36]{S}, which gives
\[\sum_{n\ge1, k\ge 1}N(n,k) u^n q^k=q\,(\mathbf{C}_q(u)-1).\] 
Therefore $C_r(q)$ is the Narayana polynomial
\[C_r(q)=\sum_{k=1}^{r} \frac{1}{r} \binom{r}{k-1} \binom{r}{k} q^{k-1}.\]

In this section, we shall derive some alternative expressions of $C_r(q)$ in a unified way. The latter are very well known, but our proof is quite elementary. It uses the binomial expansion
\[(1+a)^{1/2}=1-2\sum_{i\ge0}\binom{2i}{i}\Big(\frac{-1}{4}\Big)^{i+1}\frac{a^{i+1}}{i+1},\]
together with four possible ways of writing $(1-q)^2u^2-2(1+q)u+1$.

Actually with $\eta$ and $\zeta$ equal to $\pm 1$, we have
\[(1-q)^2u^2-2(1+q)u+1=(1+(\eta+\zeta q)u)^2\left(1-2u\frac{1+\eta+q(1+\zeta)+qu(1+\eta\zeta)}{(1+(\eta+\zeta q)u)^2}\right),\] 
which yields
\begin{align*}
\mathbf{C}_q(u)&+
\frac{1+\eta-q(1-\zeta)}{2q}=\\
&q^{-1}
\sum_{i\ge0}\binom{2i}{i}2^{-i-1}
\frac{u^i}{i+1}\Big(1+\eta+q(1+\zeta)+qu(1+\eta\zeta)\Big)^{i+1}\Big(1+(\eta+\zeta q)u\Big)^{-2i-1}.
\end{align*}
Expanding two times the right-hand side by the binomial formula, it becomes
\[q^{-1}
\sum_{i,j,k\ge0}\binom{2i}{i}{2}^{-i-1}\frac{u^{2i+j-k+1}}{i+1}
\binom{i+1}{k}(1+\eta+q(1+\zeta))^{k}
(q(1+\eta\zeta))^{i+1-k}
(-\eta-\zeta q)^j \binom{2i+j}{j}.\]
This implies
\begin{multline}\label{equn}
C_r(q)=
\sum_{\begin{subarray}{c}i,j\ge 0 \\ i+j \le r\end{subarray}}
{2}^{-i-1}q^{r-i-j-1} 
(1+\eta+q(1+\zeta))^{2i+j-r+1}\,(1+\eta\zeta)^{r-i-j}\,
(-\eta-\zeta q)^j\\
\binom{i+1}{r-i-j}\binom{2i+j}{j}
\frac{1}{i+1}\binom{2i}{i}.
\end{multline}

In a second step we specialize $\eta$ and $\zeta$ to $\pm 1$. For $\eta=-\zeta=1$ we obtain 
\begin{equation}\label{eqde}
\mathcal{C}_r(q)=q\,C_r(q)=\frac{1}{r+1}\sum_{m=0}^r  
(q-1)^{m}\binom{r+1}{m}\binom{2r-m}{r}.
\end{equation}
Since $C_r(q)$ is a polynomial in $q$, the right-hand side is divisible by $q$, namely
\[\sum_{m=0}^r  (-1)^{m}\binom{r+1}{m}\binom{2r-m}{r}=0.\]
Hence \eqref{eqde} can be alternatively written as
\begin{equation}\label{eqtr}
C_r(q)=\frac{1}{r+1}\sum_{m=0}^r (-1)^{m-1} \frac{1-(1-q)^m}{1-(1-q)}\binom{r+1}{m}\binom{2r-m}{r}.
\end{equation}
If we apply the identity
\[\sum_{m=k+1}^r  (-1)^{m-1}\binom{r+1}{m}\binom{2r-m}{r}=(-1)^k \binom{r-1}{k}\binom{2r-k}{r},\]
proved by induction on $k$, we obtain the equivalence of \eqref{eqtr} with
\begin{equation}\label{eqqu}
C_r(q)=\frac{1}{r+1}\sum_{m=0}^{r-1} (q-1)^m \binom{r-1}{m}\binom{2r-m}{r}.
\end{equation}

For $\eta=-\zeta=-1$ relation \eqref{equn} may be written as
\begin{equation}\label{eqci}
C_r(q)=\sum_{m=0}^{r}q^{m}(1-q)^{r-m}\binom{r+m}{2m}C_m.
\end{equation}
For $\eta=\zeta=-1$ it yields
\begin{equation}\label{eqsi}
C_r(q)=\sum_{m\ge0} q^m (q+1)^{r-2m-1} \binom{r-1}{2m}C_m.
\end{equation}
Finally for $\eta=\zeta=1$ we obtain 
\begin{equation*}
C_r(q)=\sum_{\begin{subarray}{c}i,j\ge 0 \\ i+j \le r\end{subarray}} (-1)^j q^{r-i-j-1} (q+1)^{2i+2j-r+1} \binom{i+1}{r-i-j}\binom{2i+j}{j}C_i.
\end{equation*}
Comparing with \eqref{eqsi}, by identification of the coefficients of $q^m (q+1)^{r-2m-1}$, and writing $r= m+n+1$, we get the interesting identity
\begin{equation*}
\sum_{i=m}^n (-1)^{n-i}
\binom{n+i}{2i}\binom{i+1}{m+1}C_i=\binom{m+n}{2m}C_m.
\end{equation*}
A direct proof might be obtained by using
\[\sum_{i=m}^n (-1)^{n-i}\binom{n+i}{i-m}\binom{n}{i}=\binom{n}{m},\] 
a variant of the Chu-Vandermonde identity. The former identity written for $m=n-1$ or $m=n-2$ gives the classical recursive formula
\[C_{r+1}=2\frac{2r+1}{r+2}C_r.\]

Relations \eqref{eqde} and \eqref{eqsi} are respectively (1.3) and (1.1) of~\cite{MS}, where references to other proofs are given. All identities are trivial for $q=1$, except \eqref{eqsi} which gives a quick proof of the celebrated Touchard's identity~\cite[p. 319]{K}
\[C_r=\sum_{m\ge0} 2^{r-2m-1} \binom{r-1}{2m}C_m.\]

\section{A $\la$-ring exercise}

Here we only give a short survey of $\la$-ring theory. More details and other applications may be found in~\cite{Las,LL,La}, or (but not explicitly) in some examples of~\cite{Ma} (see pp. 25, 43, 65 and 79).

Let $\mathbb{S}$ denote the ring of symmetric functions. The classical bases of elementary functions $e_k$, complete functions $h_{k}$, and 
power sums $p_k$ generate $\mathbb{S}$ algebraically. Schur functions $s_\mu$, monomial symmetric functions $m_\mu$, and products $e_\mu, h_\mu, p_\mu$ are linear bases of $\mathbb{S}$, indexed by partitions.

Let $X=\{x_1,x_2,x_3,\ldots\}$ be a (finite or infinite) set of
independent indeterminates (an ``alphabet''). We define an action $f \rightarrow f[\,\cdot\,]$ of $\mathbb{S}$ on the ring 
$\mathbb{R}[X]$ of polynomials in $X$ with real coefficients. Since the power sums generate $\mathbb{S}$, it is enough to define the action of $p_{k}$ on $\mathbb{R}[X]$. Writing any polynomial as 
$\sum_{c,P} c P$, with $c$ a ``constant'' and $P$ a monomial 
in ``variables'', we define
$$p_{k} \left[\sum_{c,P} c P\right]=\sum_{c,P} c P^{k}.$$

Of course this action is strongly dependent on the status of any indeterminate, which may be chosen as a ``constant'' or a ``variable'', or a combination of both. Therefore each status must be carefully specified. For instance, if $Q=1-q$ is a variable, we have $p_k[q]=1-Q^k=1-(1-q)^k$, but not $p_k[q]=q^k$. Variables are referred to as ``elements of rank 1''.

If all elements of $X$ are of rank 1, we write $X^{\dag}=\sum_{i} x_i$. By definition we have
\[p_{k} [X^{\dag}]=\sum_{i} x_i^k=p_{k}(X), \quad \mathrm{hence}
\quad f[X^\dag]=f(X)\]
for any symmetric function $f$. Similarly $f[m]=f(1,\ldots,1)$, the value of $f$ at the $m$-vector $(1,\ldots,1)$.

This action of $\mathbb{S}$ on $\mathbb{R}[X]$ has two fundamental properties. Firstly for any polynomials $P,Q$ we have
\begin{equation}\label{eqse}
h_n[P+Q]= \sum_{k=0}^n h_{n-k} [P] \ h_k [Q].
\end{equation}
Denoting the generating series by $H_u=\sum_{k\ge 0} u^k h_k$, this can be also written as
\begin{equation}\label{eqhu}
H_u[P+Q]=H_u[P]\,H_u[Q], \qquad H_u[P-Q]=H_u[P]\,{H_u[Q]}^{-1}.
\end{equation}
Secondly we have three ``Cauchy formulas''
\begin{align*}
h_n [PQ]=& \sum_{|\mu|=n} 
z_{\mu}^{-1}p_{\mu} [P] \, p_{\mu} [Q]\\
=&\sum_{|\mu|= n} h_{\mu} [P] \, m_{\mu} [Q]\\
=&\sum_{|\mu|= n} s_{\mu} [P] \, s_{\mu} [Q],
\end{align*}
where as usual, we denote by $|\mu|$ the weight of a partition, $l(\mu)$ and $\{m_i(\mu), i\ge 1\}$ its length and multiplicities, and $z_\mu  = \prod_{i} i^{m_i(\mu)} m_i(\mu) !$.
 
\begin{lem}
(i) For a constant $c$ we have
\[p_k[c]=c, \quad h_k[c]= \binom{c+k-1}{k},\quad e_k[c]= \binom{c}{k},
\quad m_\mu[c]=\binom{c}{l(\mu)}\frac{l(\mu)!}{\prod_i m_i(\mu)!}.\]
(ii) If $q$ is an element of rank 1, we have
\[p_k[q]= h_k[q]=q^k, \qquad e_k[q]=q\,\delta_{k1}.\]
(iii) If $1-q$ is an element of rank 1 we have 
\begin{align*}
p_k[q]=1-(1-q)^k,&\qquad p_k[-q]=(1-q)^k-1, \\
h_k[q]=(-1)^ke_k[-q]= q,&\qquad e_k[q]=(-1)^kh_k[-q]= q(q-1)^{k-1}.
\end{align*}
Moreover we have $s_{\mu}[q]=0$ and $s_{\mu}[-q]=0$, except if the partition $\mu$ is a hook $(a,1^b)$ in which case $s_{\mu}[q]=q(q-1)^b$ and $s_{\mu}[-q]=(-1)^{b+1}q(1-q)^{a-1}$.
\end{lem}
\begin{proof}Denoting $E_u=\sum_{k\ge 0} u^k e_k$, we have~\cite[p.25]{Ma} $E_u=(H_{-u})^{-1}$ and
\[h_k=\sum_{|\mu|=k}z_{\mu}^{-1}p_\mu,\qquad 
e_k=\sum_{|\mu|=k} (-1)^{k-l(\mu)} z_{\mu}^{-1}p_\mu.\]
If $c$ is a constant, one has~\cite[Example 1.2.1, p.26]{Ma}
\[h_k[c]=\sum_{|\mu|=k}z_{\mu}^{-1}c^{l(\mu)}=\binom{c+k-1}{k}.\]
Hence $H_u[c]=(1-u)^{-c}$ and $E_u[c]=(1+u)^{c}$, i.e. $e_k[c]= \binom{c}{k}$. The second Cauchy formula implies
\[h_k [c]=\sum_{|\mu|= k} h_{\mu} [1] \, m_{\mu} [c]=
\sum_{|\mu|= k} m_{\mu} [c].\]
The value of $m_{\mu} [c]$ follows since we have
\[\binom{c+k-1}{k}=\sum_{p=1}^k\sum_{k_1+\ldots+k_p=k} \binom{c}{p},\]
where the second sum is taken on sequences made of $p$ (strictly) positive integers summing to $k$ (compositions of length $p$ and weight $k$).

If $q$ is of rank 1, we have
\[h_k[q]=\sum_{|\mu|=k}z_{\mu}^{-1}q^{|\mu|}=q^k.\]
Hence $H_u[q]=(1-uq)^{-1}$ and $E_u[q]=1+uq$. If $Q=1-q$ is of rank 1, using \eqref{eqhu} one has 
\begin{align*}
H_u[q]&=H_u[1-Q]=H_u[1]H_u[Q]^{-1}=\frac{1-uQ}{1-u}=1+\frac{uq}{1-u},\\
E_u[q]&=\frac{1+u}{1+uQ}=1+\frac{uq}{1-u(q-1)}.
\end{align*}
For the evaluation of $s_\mu[1-Q]$ and $s_\mu[Q-1]$, we refer to~\cite[p.11]{Las}. They are respectively equal to $(-Q)^b(1-Q)$ and $(-1)^bQ^{a-1}(Q-1)$ for $\mu=(a,1^b)$. 
\end{proof}

We consider the one-row Hall-Littlewood symmetric function $P_r(X;q)$. As a straightforward consequence of its definition~\cite[(2.9), p.209]{Ma}, we have $P_0(X;q)=1$ and for $r\ge 1$,
\[P_r(X;q)=(1-q)^{-1} g_r(X;q)\quad \mathrm{with} \quad 
\sum_{k\ge 0} u^k g_k(X;q)= \prod_{i\ge 1} \frac{1-qux_i}{1-ux_i}.\]
If $q$ and all indeterminates $x_i$ are elements of rank 1, using \eqref{eqhu} we have  
\[\prod_{i\ge 1} \frac{1-qux_i}{1-ux_i}=H_u[X^{\dag}] H_u[qX^{\dag}]^{-1}=H_u[(1-q)X^{\dag}].\]
In other words, for $r\ge1$ we have
\[P_r(X;q)=(1-q)^{-1} h_r[(1-q)X^{\dag}].\]
\begin{theo}
For $r\ge 1$ we have
\begin{equation*}
P_r(1^{n};q)=\sum_{m=0}^{r-1}(-q)^m
\binom{r-1}{m}\binom{n+r-m-1}{r}.
\end{equation*}
\end{theo}
\begin{proof}
We must compute
\[P_r(1^n;q)=(1-q)^{-1} h_r[(1-q)n],\]
with $q$ of rank 1. Writing $q=1-Q$, by Lemma 1 (iii) together with the last Cauchy formula, we have
\[
h_r[Qn]=\sum_{m=0}^{r-1}s_{(r-m,1^{m})}[Q]s_{(r-m,1^{m})}[n]
=\sum_{m=0}^{r-1}Q(Q-1)^ms_{(r-m,1^{m})}[n].
\]
The assertion is then a consequence of 
\begin{equation}\label{eqne}
s_{(a,1^b)}[c]=\binom{a+b-1}{b}\binom{a+c-1}{a+b},
\end{equation}
for any real number $c$. This can be proved in many ways, for instance by induction, using Lemma 1 (i) and the Pieri formula
\[h_ae_b=s_{(a,1^b)}+s_{(a+1,1^{b-1})},\]
which imply
\[s_{(a,1^b)}[c]+s_{(a+1,1^{b-1})}[c]=\binom{a+c-1}{a}\binom{c}{b}.\]
\end{proof}

We shall give two proofs of the following theorem. The first one relies on an identity given in Section 2. The second proof merely uses the definition of Narayana polynomials.
\begin{theo}
For $r\ge 1$ we have
\begin{equation*}
C_r(1-q)=\frac{1}{r+1}P_r(1^{r+1};q).
\end{equation*}
Equivalently in $\la$-ring notation, with $q$ an element of rank 1, we have
\begin{equation*}
\mathcal{C}_r(1-q)=(1-q)\,C_r(1-q)=\frac{1}{r+1}h_r[(1-q)(r+1)].
\end{equation*}
\end{theo}

\begin{proof}[First proof] Immediate consequence of \eqref{eqqu} and Theorem 1.
\end{proof} 

\begin{proof}[Second proof]
Up to now, we have only used the third Cauchy formula. But we may also use the second one, which writes as
\[P_r(1^{r+1};q)=(1-q)^{-1} h_r[(1-q)(r+1)]=
(1-q)^{-1} \sum_{|\mu|= r} h_{\mu} [1-q] \, m_{\mu} [r+1],\]
with $q$ of rank 1. Applying Lemma 1 (i) and (iii), we get
\[P_r(1^{r+1};q)=\sum_{|\mu|= r}(1-q)^{l(\mu)-1}
\binom{r+1}{l(\mu)}\frac{l(\mu)!}{\prod_i m_i(\mu)!}.\]
Thus Theorem 2 amounts to prove
\begin{align*}
C_r(1-q)&=\frac{1}{r+1} \sum_{|\mu|= r}(1-q)^{l(\mu)-1}
\binom{r+1}{l(\mu)}\frac{l(\mu)!}{\prod_i m_i(\mu)!}\\
&=\sum_{k=1}^r (1-q)^{k-1}\binom{r}{k-1}
\sum_{\begin{subarray}{c}|\mu|= r \\ l(\mu)=k\end{subarray}}
\frac{(k-1)!}{\prod_i m_i(\mu)!}.
\end{align*}
This is exactly the definition of Narayana polynomials, because one has
\[\binom{r-1}{k-1}=\sum_{\begin{subarray}{c}|\mu|= r \\ l(\mu)=k\end{subarray}}\frac{k!}{\prod_i m_i(\mu)!}.\]
The latter identity counts the number of sequences made of $k$ (strictly) positive integers summing to $r$ (compositions of length $k$ and weight $r$). See~\cite[p.25]{St} for a proof and~\cite[Theorem 1, p.461]{La3} for a generalization.
\end{proof}

As for the first Cauchy formula, with $q$ of rank 1, it implies
\begin{align*}
P_r(1^{r+1};q)&=(1-q)^{-1} h_r[(1-q)(r+1)]\\
&=(1-q)^{-1} \sum_{|\mu|= r} z_{\mu}^{-1}p_{\mu} [1-q] \, p_{\mu} [r+1]\\
&=(1-q)^{-1} \sum_{|\mu|= r} z_{\mu}^{-1}(r+1)^{l(\mu)}\prod_{i\ge 1}(1-q^i)^{m_i(\mu)}.
\end{align*}
By Theorem 2 we obtain
\[\mathcal{C}_r(q)=q\,C_r(q)=\sum_{|\mu|= r} z_{\mu}^{-1}(r+1)^{l(\mu)-1}\prod_{i\ge 1}(1-(1-q)^i)^{m_i(\mu)}.\]
This formula is new. It is obvious for $q=1$, since by applying Lemma 1 (i) we have
\begin{equation*}
h_r[r+1]=\sum_{|\mu|=r}z_{\mu}^{-1}(r+1)^{l(\mu)}=\binom{2r}{r}=(r+1)C_r.
\end{equation*}
For $q=2$ it yields the following interesting expression for the small Schr\"{o}der numbers
\[
s_r=\sum_{\begin{subarray}{c}|\mu|= r \\ \mathrm{all}\,\mathrm{parts}\,\mathrm{odd}\end{subarray}} z_{\mu}^{-1}(2r+2)^{l(\mu)-1}.\]

\noindent\textit{Remark on Lagrange involution: }An involution $f \rightarrow f^{*}$ can be defined on $\mathbb{S}$ as follows (\cite[Example 1.2.24, p. 35]{Ma}, \cite[Section 2.4]{Las}). Let
\[u=tH_t=\sum_{k\geq0} t^{k+1}\, h_k.\]
Then $t$ can be expressed as a power series in $u$, its compositional inverse, namely
\[t=uH_u^{*}=\sum_{k\geq0} u^{k+1}\, h^{*}_k.\]
The map $h_k \rightarrow h_k^{*}$ extends to an involution of $\mathbb{S}$, called ``Lagrange involution''. Many identities are obtained in this context (see~\cite[Section 4, p.2236]{La2} for a detailed account). In particular for any polynomial $A$ we have~\cite[(4.10)]{La2}
\[(r+1)h_r^{*}[A]=h_r[-(r+1)A].\]
Therefore Theorem 2 can be equivalently written as
\[\mathcal{C}_r(1-q)=(1-q)\,C_r(1-q)=h_r^{*}[q-1],\]
with $q$ an element of rank 1.

\section{A generalization of Koshy's formula}

The following sections present some examples of the efficiency of $\la$-ring calculus. We start from a remarkable $\la$-ring identity, already mentioned in~\cite[(4.14), p.2239]{La2}.
\begin{lem}
For any polynomial $A$, any real number $z$ and any integer $n\ge1$ we have
\[\sum_{k=0}^n \frac{1}{z+k} h_{k}[-(z+k)A] \,h_{n-k}[(z+k)A]=0.\]
\end{lem}
\begin{proof}
We apply Lemma 1 (i) and the second Cauchy formula to get
\[h_{k}[-(z+k)A]=\sum_{|\mu|= k} m_{\mu} [-z-k]\,h_{\mu} [A]=
\sum_{|\mu|= k} \binom{-z-k}{l(\mu)}\frac{l(\mu)!}{\prod_i m_i(\mu)!}\,h_{\mu} [A],\]
and similarly
\[h_{n-k}[(z+k)A]=\sum_{|\nu|= n-k} \binom{z+k}
{l(\nu)}\frac{l(\nu)!}{\prod_i m_i(\nu)!}\,h_{\nu} [A].\]
Therefore for any partition $\rho$, it is equivalent to prove
\begin{equation}\label{rot}
\sum_{\mu\cup\nu=\rho}\frac{l(\mu)!}{\prod_i m_i(\mu)!} \,
\frac{l(\nu)!}{\prod_i m_i(\nu)!} \frac{1}{z+|\mu|}\binom{-z-|\mu|}{l(\mu)} \binom{z+|\mu|}{l(\nu)} =0,
\end{equation}
the sum taken over all decompositions of $\rho$ into two partitions (possibly empty). This relation may be written as
\[\sum_{\mu\cup\nu=\rho}\frac{1}{\prod_i m_i(\mu)!m_i(\nu)!} \,
(-1)^{l(\mu)}\prod_{j=-l(\nu)+1}^{l(\mu)-1}(z+|\mu|+j)=0.\]
Equivalently
\[\sum_{\mu\cup\nu=\rho}\prod_i \binom{m_i(\rho)}{ m_i(\mu)} \,
 (-1)^{l(\mu)}\prod_{j=1}^{l(\rho)-1}(z+|\mu|-l(\nu)+j)=0.\]

The proof of this identity is done in two steps. Firstly we observe that we may restrict to partitions $\rho$ having only parts with multiplicity $1$, i.e.
\[\sum_{\mu\cup\nu=\rho}
 (-1)^{l(\mu)}\prod_{j=1}^{l(\rho)-1}(z+|\mu|-l(\nu)+j)=0.\]
Actually given a partition $\rho=\mu\cup\nu$ having only parts with multiplicity $1$, there are $\binom{m_i(\rho)}{ m_i(\mu)}$ contributions becoming equal when $m_i(\rho)$ parts get equal to $i$.

In a second step, we observe that the previous identity may be generalized as
\[\sum_{A\subset E}(-1)^{\mathrm{card}\, A} \prod_{i=1}^{\mathrm{card}\, E-1}(z+|A|-\mathrm{card}\, E+\mathrm{card}\, A+i)=0,\]
where $E$ is some set of indeterminates and $|A|=\sum_{a\in A}a$. This is equivalent to
\[\sum_{A\subset E}(-1)^{\mathrm{card}\, A} \prod_{i=1}^{\mathrm{card}\, E-1}(z+|A|+i)=0,\]
because we can change $z$ into $z+\mathrm{card}\, E$, and any $a\in E$ into $a-1$ so that $|A|+\mathrm{card}\, A$ is changed into $|A|$. Finally we apply the following lemma with $y_i=z+i$.

\end{proof}

\begin{lem}
Let $E=\{x_1,\ldots,x_n\}$ be a set of $n$ indeterminates. For any indeterminates $y=(y_1,\ldots,y_n)$ we have
\begin{align*}
\sum_{A\subset E}(-1)^{n-\mathrm{card}\, A} \prod_{i=1}^{n}(|A|+y_i)&=n! \prod_{i=1}^{n}x_i,\\
\sum_{A\subset E}(-1)^{n-\mathrm{card}\, A} \prod_{i=1}^{n-1}(|A|+y_i)&=0.
\end{align*}
\end{lem}
\begin{proof}We consider the monomial symmetric functions in the indeterminates $(x_1,\ldots,x_n)$. For any partition $\mu$ with length $\le n-1$ and weight $\le n$ we have
\[\sum_{A\subset \{x_1,\ldots,x_n\}}(-1)^{n-\mathrm{card}\, A} m_\mu(A)=0.\]
This is proved by induction on $n$, the sum being evaluated as
\[\sum_{A\subset \{x_1,\ldots,x_{n-1}\}}(-1)^{n-\mathrm{card}\, A} m_\mu(A)
-\sum_{k=1}^{n} x_n^k \sum_{A\subset \{x_1,\ldots,x_{n-1}\}}(-1)^{n-\mathrm{card}\, A} m_{\mu\setminus k}(A),\]
where $\mu\setminus k$ denotes the partition obtained by subtracting the part $k$ of $\mu$ (if it exists). Since for $k\le n-1$ the function $e_1^k$ is formed of monomial symmetric functions of weight $\le n-1$, we get
\[\sum_{A\subset \{x_1,\ldots,x_n\}}(-1)^{n-\mathrm{card}\, A} |A|^k=0\qquad \quad (k=0,\ldots,n-1).\]
For $\mu=1^n$ we have directly
\[\sum_{A\subset \{x_1,\ldots,x_n\}}(-1)^{n-\mathrm{card}\, A} m_{1^n}(A)=m_{1^n}(E)=\prod_{i=1}^{n}x_i.\]
Since $e_1^n=n!\,m_{1^n}$ up to monomial symmetric functions of weight $n$ and length $\le n-1$, we get
\[\sum_{A\subset \{x_1,\ldots,x_n\}}(-1)^{n-\mathrm{card}\, A} |A|^n=n! \,\prod_{i=1}^{n}x_i.\]
We conclude by writing the expansions
\begin{align*}
\prod_{i=1}^{n}(|A|+y_i)&=\sum_{k=0}^n |A|^k\, e_{n-k}(y),\\
\prod_{i=1}^{n-1}(|A|+y_i)&=\sum_{k=0}^{n-1} |A|^k\, e_{n-k-1}(y).
\end{align*}
in terms of the symmetric functions of $y$.
\end{proof}

\noindent\textit{Remark: }  Equation~\eqref{rot} is a generalization (in the framework of partitions) of the classical Rothe identity~\cite{Chu} 
\[\sum_{k=0}^n\frac{x}{x-k}\binom{x-k}{k}\binom{y+k}{n-k}=\binom{x+y}{n},\]
taken at $x=-y$.

In view of Lemma 1 (i) we have
\[\frac{1}{k+1} h_{k}[k+1]=C_k,\qquad
h_{n-k}[-(k+1)]=\binom{n-2k-2}{n-k}=(-1)^{n-k}\binom{k+1}{n-k}.\]
Thus Koshy's formula~\cite[p.322]{K}
\[C_n=\sum_{k=1}^n (-1)^{k-1}\binom{n-k+1}{k}C_{n-k}\]
can be written as
\[\sum_{k=0}^n (-1)^{n-k}\binom{k+1}{n-k}C_{k}=\sum_{k=0}^n \frac{1}{k+1} h_{k}[k+1] \,h_{n-k}[-(k+1)]=0,\]
which is the case $z=1, A=-1$ of Lemma 2. We only need to change $A=-1$ into $A=-q$ to obtain the following generalization.
\begin{theo}
Narayana polynomials satisfy the recurrence relation
\[C_n(q)=(1-q)^{n-1}+q\sum_{k=1}^{n-1}C_{n-k}(q)
\sum_{m=0}^{k-1} (-1)^m\binom{k-1}{m}\binom{n-m}{k} (1-q)^{k-m-1}.\]
\end{theo}
\begin{proof}
If we apply Lemma 2 with $z=1$ and $A=-q$, we obtain
\[\sum_{k=0}^n \frac{1}{k+1} h_{k}[(k+1)q] \,h_{n-k}[-(k+1)q]=0.\]
If we assume $1-q$ to be an element of rank 1, Theorem 2 may be written as
\[\frac{1}{k+1} h_{k}[(k+1)q]=q\,C_k(q)\qquad (k\ge1).\]
Hence the previous relation becomes
\[q^{-1}h_n[-q]+\sum_{k=1}^{n-1} C_k(q) \,h_{n-k}[-(k+1)q]+C_n(q)=0.\]
Applying Lemma 1 (iii) and the last Cauchy formula, we have
\begin{align*}
h_{n-k}[(k+1)(-q)]&=\sum_{m=0}^{n-k-1}s_{(n-k-m,1^{m})}[-q]\,s_{(n-k-m,1^{m})}[k+1]\\
&=\sum_{m=0}^{n-k-1}(-1)^{m+1}q(1-q)^{n-k-m-1}s_{(n-k-m,1^{m})}[k+1]\\
&=-q\sum_{m=0}^{n-k-1}(-1)^{m}(1-q)^{n-k-m-1} \binom{n-k-1}{m}\binom{n-m}{n-k},
\end{align*}
where the last equation is a consequence of \eqref{eqne}.
In particular for $k=0$, we have $h_n[-q]=-q(1-q)^{n-1}$. Changing $k$ to $n-k$, we can conclude.
\end{proof}

For $q=1$, we recover Koshy's formula. For $q=2$ we obtain the following recurrence for the small Schr\"{o}der numbers (which seems to be new)
\[s_n=(-1)^{n-1}+2\sum_{k=1}^{n-1}(-1)^{k-1} s_{n-k}
\sum_{m=0}^{k-1} \binom{k-1}{m}\binom{n-m}{k}.\]

\section{Two generalizations of Jonah's formula}

Lemma 2 is the particular case $B=0$ of the following $\la$-ring identity.
\begin{lem}
For any polynomials $A$ and $B$, any real number $z$ and any integer $n\ge1$ we have
\[\sum_{k=0}^n \frac{z}{z+k} h_{k}[-(z+k)A] \,
h_{n-k}[(z+k)A+B]=h_n[B].\]
\end{lem}
\begin{proof}Applying \eqref{eqse} we have
\[ h_{n-k}[(z+k)A+B]=\sum_{j=0}^{n-k} h_{n-k-j}[B] \,h_{j}[(z+k)A],\]
which yields
\begin{multline*}
\sum_{k=0}^n \frac{z}{z+k} h_{k}[-(z+k)A] \,h_{n-k}[(z+k)A+B]\\=
\sum_{i=0}^n h_{n-i}[B]\sum_{k=0}^i \frac{z}{z+k} h_{k}[-(z+k)A] \,h_{i-k}[(z+k)A].
\end{multline*}
We conclude by applying Lemma 2 for $i\neq0$.
\end{proof}

In view of Lemma 1 (i) we have
\[\frac{1}{k+1} h_{k}[k+1]=C_k,\qquad
h_{r-k}[n-r-k+1]=\binom{n-2k}{r-k},\qquad
h_r[n-r+2]=\binom{n+1}{r}.\]
Thus Jonah's formula~\cite[p.325]{K}
\[\sum_{k=0}^r\binom{n-2k}{r-k}C_{k}=\binom{n+1}{r}\]
can be written as
\[\sum_{k=0}^r \frac{1}{k+1} h_{k}[k+1] \,
h_{r-k}[n-r-k+1]=h_r[n-r+2].\]
This is the case $z=1, A=-1, B= n-r+2$ of Lemma 4. Our first generalization only needs to change $A=-1, B= n-r+2$ into $A=-q, B= q(n-r+2)$.
\begin{theo}
For any positive integers $n,r$ we have
\begin{multline*}
C_r(q)+q\sum_{k=1}^{r-1} C_{r-k}(q) 
\sum_{m=0}^{k-1}(q-1)^{m}\binom{k-1}{m}\binom{n-2r+2k-m}{k}
\\
=\sum_{m=0}^{r-1}(q-1)^{m}\binom{r-1}{m}\binom{n-m}{r-1}.
\end{multline*}
\end{theo}
\begin{proof}
If we apply Lemma 4 with $z=1$, $B= q(n-r+2)$ and $A=-q$ we obtain
\[\sum_{k=0}^r \frac{1}{k+1} h_{k}[(k+1)q] \,h_{r-k}[(n-r-k+1)q]=h_r[(n-r+2)q].\]
Assuming $1-q$ to be an element of rank 1 and applying Theorem 2, this relation becomes
\[
q\,\sum_{k=1}^{r-1} C_k(q) 
\,h_{r-k}[(n-r-k+1)q]+q\,C_r(q)=h_r[(n-r+2)q]-h_r[(n-r+1)q].
\]
Applying Lemma 1 (iii) and the last Cauchy formula, we have
\begin{align*}
h_{r-k}[(n-r-k+1)q]&=\sum_{m=0}^{r-k-1}s_{(r-k-m,1^{m})}[q]\,s_{(r-k-m,1^{m})}[n-r-k+1]\\
&=\sum_{m=0}^{r-k-1}q(q-1)^m s_{(r-k-m,1^{m})}[n-r-k+1]\\
&=q\sum_{m=0}^{r-k-1}(q-1)^{m}\binom{r-k-1}{m}\binom{n-2k-m}{r-k},
\end{align*}
where the last equation is a consequence of \eqref{eqne}. Exactly in the same way we have
\[h_r[(n-r+2)q]-h_r[(n-r+1)q]=q\sum_{m=0}^{r-1}(q-1)^{m}\binom{r-1}{m}\binom{n-m}{r-1}.\]
Summing the contributions and changing $k$ to $r-k$, we can conclude.
\end{proof}

For $q=1$ we recover Jonah's formula under the form
\[\sum_{k=1}^r\binom{n-2k}{r-k}C_{k}=\binom{n}{r-1}.\]
For $q=2$ we obtain the following identity for the small Schr\"{o}der numbers (which is probably new)
\[s_r+2\sum_{k=1}^{r-1} s_{r-k} 
\sum_{m=0}^{k-1}\binom{k-1}{m}\binom{n-2r+2k-m}{k}
=\sum_{m=0}^{r-1}\binom{r-1}{m}\binom{n-m}{r-1}.\]

Our second generalization of Jonah's formula only needs to change $A=-1, B= n-r+2$ into $A=-q, B= n-r+2$.
\begin{theo}
For any positive integers $n,r$ we have
\[\sum_{k=0}^r \mathcal{C}_k(q) \sum_{m=0}^{r-k}(1-q)^m \binom{n-2k-m}{r-k-m}\binom{k+m}{m}=\binom{n+1}{r}.\]
\end{theo}
\begin{proof}
If we apply Lemma 4 with $z=1, B= n-r+2$ and $A=-q$ we obtain
\[\sum_{k=0}^r \frac{1}{k+1} h_{k}[(k+1)q] \,h_{r-k}[(n-r-k+1)+(k+1)(1-q)]=\binom{n+1}{r}.\]
Assuming $1-q$ to be an element of rank 1 and applying Theorem 2, the left-hand side becomes
\[
h_{r}[(n-r+1)+1-q]+q\sum_{k=1}^{r} C_k(q) 
\,h_{r-k}[(n-r-k+1)+(k+1)(1-q)].
\]
By \eqref{eqse} we have
\[
h_{r-k}[(n-r-k+1)+(k+1)(1-q))]=\sum_{m=0}^{r-k}h_{r-k-m}[n-r-k+1]\,h_{m}[(k+1)(1-q)].\]
By Lemma 1 (i) we get
\[h_{r-k-m}[n-r-k+1]=\binom{n-2k-m}{r-k-m}.\]
On the other hand, since $1-q$ is an element of rank 1, we have
\begin{align}\label{eqdi}
H_u[(k+1)(1-q)]&=\Big(H_u[1-q]\Big)^{k+1}\notag\\
&=\big(1-u(1-q)\big)^{-k-1}\\
&=\sum_{m\ge 0} u^m (1-q)^m \binom{k+m}{m},\notag
\end{align}
where the second relation is a consequence of Lemma 1 (ii). Finally we obtain
\[h_{r-k}[(n-r-k+1)+(k+1)(1-q))]=\sum_{m=0}^{r-k}\binom{n-2k-m}{r-k-m}\,(1-q)^m \binom{k+m}{m}.\]
\end{proof}

For $q=2$ we obtain the following identity for the large Schr\"{o}der numbers (which seems to be new)
\[\sum_{k=0}^r R_k \sum_{m=0}^{r-k}(-1)^m \binom{n-2k-m}{r-k-m}\binom{k+m}{m}=\binom{n+1}{r}.\]
For a very different extension of Koshy's and Jonah's identities, see~\cite{A}.

\section{Transition matrix}

\begin{theo}
For any positive integer $n$ we have
\[
(n+1)\mathcal{C}_n(q^{\prime})=\sum_{k=0}^n \mathcal{C}_{n-k}(q)\sum_{i+j=0}^{k}
(1-q)^i(q^{\prime}-1)^j \binom{n-k+i}{i} \binom{n+1}{j} \binom{2k-i-j-1}{k-i-j}.
\]
\end{theo}
\begin{proof}
If we apply Lemma 4 with $z=1$, $A=-q$ and $B=(n+1)q^{\prime}$, we obtain
\[
\sum_{k=0}^n \frac{1}{k+1} h_{k}[(k+1)q] \,h_{n-k}[(n-k)+(k+1)(1-q)-(n+1)(1-q^{\prime})]
=h_n[(n+1)q^{\prime}].
\]
Assuming $1-q$ and $1-q^{\prime}$ of rank 1 and applying Theorem 2, the right-hand side is $(n+1)q^{\prime}\,C_n(q^{\prime})$ and the left-hand side becomes
\[
h_{n}[n+(1-q)-(n+1)(1-q^{\prime})]+q\,\sum_{k=1}^{n} C_k(q) 
\,h_{n-k}[(n-k)+(k+1)(1-q)-(n+1)(1-q^{\prime})].
\]
By \eqref{eqse} we have
\begin{align*}
h_{n-k}[(n-k)&+(k+1)(1-q)-(n+1)(1-q^{\prime})]\\
&=\sum_{i+j=0}^{n-k}h_{n-k-i-j}[n-k]\,h_{i}[(k+1)(1-q)]\,h_{j}[-(n+1)(1-q^{\prime})]\\
&=\sum_{i+j=0}^{n-k}\binom{2n-2k-i-j-1}{n-k-i-j}\,(1-q)^i \binom{k+i}{i}
\,(q^{\prime}-1)^j \binom{n+1}{j}.
\end{align*}
Here the last relation is a consequence of
Lemma 1 (i), together with \eqref{eqdi} and
\begin{align*}
H_u[-(n+1)(1-q^{\prime})]&=\Big(H_u[1-q^{\prime}]\Big)^{-n-1}\\
&=\big(1-u(1-q^{\prime})\big)^{n+1}\\
&=\sum_{j\ge 0} u^j (q^{\prime}-1)^j \binom{n+1}{j},
\end{align*}
since $1-q^{\prime}$ is an element of rank 1. Changing $k$ to $n-k$, we can conclude.
\end{proof}

Making $q=1$ or $q^\prime=1$, we obtain
\[(n+1)\mathcal{C}_n(q)=\sum_{k=0}^n C_{n-k}\sum_{j=0}^{k}
(q-1)^j \binom{n+1}{j} \binom{2k-j-1}{k-j},\]
\[(n+1)C_n=\sum_{k=0}^n\mathcal{C}_{n-k}(q)\sum_{i=0}^{k}
(1-q)^i\binom{n-k+i}{i} \binom{2k-i-1}{k-i},\]
which can be also transformed to known results (respectively (1.3) and (1.4) of~\cite{MS}). Making $q=2$ or $q^\prime=2$, analogous identities are
\[
(n+1)\mathcal{C}_n(q)=\sum_{k=0}^n R_{n-k}\sum_{i+j=0}^{k}
(-1)^i(q-1)^j \binom{n-k+i}{i} \binom{n+1}{j} \binom{2k-i-j-1}{k-i-j},
\] 
\[
(n+1)R_n=\sum_{k=0}^n \mathcal{C}_{n-k}(q)\sum_{i+j=0}^{k}
(1-q)^i \binom{n-k+i}{i} \binom{n+1}{j} \binom{2k-i-j-1}{k-i-j},
\]
which connect Narayana polynomials with the large Schr\"{o}der numbers.

\section{Narayana alphabet}

Since the complete symmetric functions are algebraically independent, they may be specialized in any way. Given a family of functions $\{f_n, n\ge 0\}$, it is possible to write $f_n=h_n(A)$ for some (at least formal) alphabet $A$, provided $f_0=1$. Equivalently the generating function for the $f_n$'s is then $H_u(A)=\sum_{n\ge 0} u^n h_n(A)$.

Let us perform such a specialization for the Narayana polynomials and denote $A$ the ``Narayana alphabet'' defined by $h_n(A)=C_n(q)$. Equivalently we have $H_u(A)=\mathbf{C}_q(u)$. Similarly we denote $A_1$ the ``Catalan alphabet'' defined by $h_n(A_1)=C_n$ and $H_u(A_1)=H_u(A)\arrowvert_{q=1}$.

We may compute some classical bases of symmetric functions for $A$, for instance the power sums $\{p_r(A), r\ge 1\}$ or the Schur functions $s_\mu(A)$. 

\begin{theo}
We have 
\[s_{k^k}(A)=s_{(k-1)^k}(A)=(-q)^{\binom{k}{2}}.\]
\end{theo}
\begin{proof}
Equation \eqref{equnun} implies
\begin{equation}\label{eqseun}
1-\frac{1}{\mathbf{C}_q(u)}=u(1-q+q\mathbf{C}_q(u))=\frac{u}{1-qu\mathbf{C}_q(u)}.
\end{equation}
Therefore the generating function $\mathbf{C}_q(u)$ may be written as
\begin{equation*}
\mathbf{C}_q(u)=\frac{1}{1-\frac{\displaystyle{u}}{\displaystyle{1-qu\mathbf{C}_q(u)}}},
\end{equation*}
which yields a very simple expression as the continued fraction~\cite[Section 3.5]{B} 
\[\mathbf{C}_q(u)=\cfrac{1}{1- \cfrac{u}
{1- \cfrac{qu}{1- \cfrac{u}{1- \cfrac{qu}{1-\ddots}}}}},\]
whose coefficients are alternatively $1$ and $q$.
We then apply~\cite[(5.3.5)]{Las} according which, in the expression of $H_u(A)$ as a continued fraction, the coefficients are given in terms of $s_{k^k}(A)$ and $s_{(k-1)^k}(A)$.
\end{proof}

By computer calculations, Lascoux noticed that the power sums $\{p_r(A), r\ge 1\}$ are polynomials in $q$ with positive integer coefficients. For instance
\begin{align*}
p_1(A)=1,\quad p_2(A)=2 q + 1,&\quad p_3(A)=3 q^2+ 6 q + 1,\\ 
p_4(A)= 4 q^3 + 18 q^2  + 12 q + 1, &\quad p_5(A)= 5 q^4 + 40 q^3  + 60 q^2  + 20 q + 1.
\end{align*}
\begin{theo}
The power sum $p_r(A)$ is given by
\[p_r(A)=\sum_{k=0}^{r-1} \binom{r-1}{k}\binom{r}{k}q^{k}.\]
\end{theo}
\begin{proof}
By a classical formula~\cite[p. 23]{Ma}, we have
\[
\sum_{r\ge 1}p_r(A)u^{r-1}=\frac{d/du \,\mathbf{C}_q(u)}{\mathbf{C}_q(u)}.\]
Using \eqref{equnde}, by explicit computation the right-hand side is 
\[
\frac{1}{\sqrt{(1-q)^2u^2-2(1+q)u+1}}
\Big(\frac{1}{u}-\frac{2q}{(1-(1-q)u-\sqrt{(1-q)^2u^2-2(1+q)u+1}}\Big).\]
In view of \eqref{eqseun}, this can be transformed into
\begin{align*}
\frac{1}{u\sqrt{(1-q)^2u^2-2(1+q)u+1}}\Big(1-\frac{1}{\mathbf{C}_q(u)}\Big)
&=\frac{1+q(\mathbf{C}_q(u)-1)}{\sqrt{(1-q)^2u^2-2(1+q)u+1}}\\
&=\frac{-1}{2u}\Big(1-\frac{1+(1-q)u}{\sqrt{(1-q)^2u^2-2(1+q)u+1}}\Big)\\
&=\frac{-1}{2u}\Big(1-\Big(1-\frac{4u}{(1+(1-q)u)^2}\Big)^{-1/2}\,\Big).
\end{align*}
Using 
\[(1-a)^{-1/2}=1+2\sum_{i\ge1}\binom{2i-1}{i}\Big(\frac{a}{4}\Big)^{i},\]
followed by the binomial expansion of $(1+(1-q)u)^{-2i}$, we get
\begin{align*}
p_r(A)&=\sum_{k=1}^r\binom{2k-1}{k}\binom{r+k-1}{r-k}(q-1)^{r-k}\\
&=\sum_{k=1}^r\binom{2k-1}{k}\binom{r+k-1}{r-k}\sum_{m=0}^{r-k}(-1)^{r-k-m}\binom{r-k}{m}q^m.
\end{align*}
Therefore it remains to prove
\[\binom{r-1}{m}\binom{r}{m}=\sum_{k=1}^r\binom{2k-1}{k}\binom{r+k-1}{r-k}(-1)^{r-k-m}\binom{r-k}{m}.\]
Since we have
\[\binom{2k-1}{k}\binom{r+k-1}{r-k}\binom{r-k}{m}=
\binom{r}{m} \binom{r+k-1}{r}\binom{r-m}{k},\]
this is equivalent to
\begin{align*}
\binom{r-1}{m}&=\sum_{k=1}^r\binom{r+k-1}{r}\binom{r-m}{k}(-1)^{r-k-m}\\
&=\sum_{k=1}^r\binom{r+k-1}{r}\binom{-k-1}{r-m-k}\\
&=\sum_{k=1}^r\binom{r+k-1}{r}\binom{-k-1}{-r+m-1},
\end{align*}
a variant of the Chu-Vandermonde identity.
\end{proof}

On the other hand, by induction, the power sum $p_r(A_1)$ may be independently shown to be \[p_r(A_1)=\binom{2r-1}{r-1}.\]
Actually the classical formula~\cite[p. 23]{Ma}
\[nh_n=\sum_{r=0}^{n-1} h_{r}p_{n-r}\]
written for the alphabet $A_1$ gives
\[\frac{n}{n+1} \binom{2n}{n}=\sum_{r=0}^{n-1} \frac{1}{r+1} h_{r}[r+1]h_{n-r-1}[n-r+1]=h_{n-1}[n+2]=\binom{2n}{n-1}.\]
This is obtained by using Lemma 4 written with $z=1, A=-1, B= n+2$ together with three applications of Lemme 1 (i), namely 
\[h_r(A_1)=\frac{1}{r+1} h_{r}[r+1], \quad p_{n-r}(A_1)=h_{n-r-1}[n-r+1], \quad h_{n-1}[n+2]=\binom{2n}{n-1}.\]

In other words, in the same way than Narayana polynomials $C_r(q)$ are a $q$-refinement of Catalan numbers $C_r=C_r(1)$ because
\[\sum_{k=1}^{r} N(r,k)=C_r,\]
the polynomials $p_r(A)$ are a refinement of $p_r(A_1)=p_r(A)\arrowvert_{q=1}$, because
\[\sum_{m=0}^{r-1} \binom{r-1}{m}\binom{r}{m}=\binom{2r-1}{r-1}.\]

Moreover Lascoux also noticed, by computer calculations, that up to a sign, the Schur functions $s_\mu(A)$ are polynomials in $q$ with positive integer coefficients. For instance
\begin{align*}
s_6=   q^5+15q^4+50q^3+50q^2+15q+1,&\quad
-s_{51}=   q^5+14q^4+40q^3+30q^2+5q,\\
-s_{42}=   3q^3+8q^2+3q,&\quad
s_{411}=   q^5+13q^4+34q^3+24q^2+4q,\\
-s_{33}=   q^3+q^2+q,&\quad
s_{321}=   4q^3+7q^2+2q,\\
-s_{31^3}= q^5+12q^4+30q^3+20q^2+3q,&\quad
-s_{2^3}=  q^3,\quad
-s_{2^21^2}= 3q^3+5q^2+q,\\
s_{21^4}=   q^5+11q^4+26q^3+16q^2+2q,&\quad
-s_{1^6}=   q^5+10q^4+20q^3+10q^2+q.
\end{align*}
It would be interesting to investigate this property, which might perhaps be interpreted in terms of some statistics.

\section{Final remarks}

Let $\mathcal{P}_n^{(1,1)}$ denote the classical Jacobi (actually Gegenbauer) polynomial defined by
\[\mathcal{P}_n^{(1,1)}(x)= \sum_{m=0}^n \binom{n+1}{m}\binom{n+1}{n-m} \Big(\frac{x-1}{2}\Big)^{n-m} \Big(\frac{x+1}{2}\Big)^{m}.\] 
In a recent paper~\cite{KK} it is proved that one has
\[C_n(x)=\frac{(x-1)^{n-1}}{n} \mathcal{P}_{n-1}^{(1,1)}\Big(\frac{x+1}{x-1}\Big).\]
Taking into account Theorem 2, this yields
\begin{equation*}
P_n(1^{n+1};q)=\frac{n+1}{n}(-q)^{n-1}\mathcal{P}_{n-1}^{(1,1)}(1-2/q).
\end{equation*}
This specialization of Hall-Littlewood polynomials seems to be new. In view of Theorem 1, it amounts to the identity
\[\sum_{m=0}^{n-1} \binom{n-1}{m}\binom{2n-m}{n}(-q)^m=
\sum_{m=0}^{n-1} \binom{n+1}{m+1}\binom{n-1}{m} (1-q)^{m},\]
which is a consequence of the Chu-Vandermonde formula.

Our second remark is devoted to generalized Narayana numbers, which have been introduced in~\cite[Section 5.2]{F} in the context of the non-crossing partition lattice for the reflection group associated with a root system. Ordinary Narayana polynomials correspond to a root system of type $A$. 

For a root system of type $B$, generalized Narayana polynomials are defined~\cite[Example 5.8]{F} by $W_0(z)=1$ and
\[W_r(z)=\sum_{k=0}^{r} {\binom{r}{k}}^2 z^{k}.\]
For their combinatorial study we refer to~\cite{Ch2,Ch} and references therein. We have $W_r(1)=W_r$, the central binomial coefficient, since
\[ W_r=\binom{2r}{r}=\sum_{k=0}^{r} {\binom{r}{k}}^2.\]
Moreover~\cite[equation (2.1)]{Ch} the Narayana polynomial $W_r(z)$ can be expressed in terms of central binomial coefficients as
\begin{equation*}
W_r(z)=\sum_{m\ge0} z^m (z+1)^{r-2m} \binom{r}{2m}W_m.
\end{equation*}

It is an open problem whether the polynomials $W_r(z)$ can be obtained by specialization of some classical symmetric function. The Hall-Littlewood polynomial of type $B$~\cite{Le} might be a good candidate.

Finally Christian Stump (private communication) pointed out the existence of a combinatorial proof of Theorem 2, which is sketched below.

Using the description of Hall-Littlewood polynomials given in~\cite{H} (see Definitions 2.1 -- 2.2 and Theorem 2.3 of~\cite{LeL}), it is known that
\[P_n(1^{n+1};q) = \sum_w (1-q)^{\mathrm{strinc}(w)}.\]
Here the sum is taken over all weakly increasing sequences $w$ having length $n$ and entries bounded by $n+1$, and $\mathrm{strinc}(w)$  is the number of strictly increasing positions of $w$. 

For instance for $n=3$ such sequences are\\
\begin{tabular}{llllllllll}
   111 & 112 & 113 & 114 & 122 & 123 & 124 & 133 & 134 & 144 \\
   222 & 223 & 224 & 233 & 234 & 244 &&&& \\
   333 & 334 & 344 &&&&&&\\
   444 &&&&&&&&
\end{tabular}\\
and the corresponding strinc statistics are\\
\begin{tabular}{llllllllll}
   0 & 1 & 1 & 1 & 1 & 2 & 2 & 1 & 2 & 1 \\
   0 & 1 & 1 & 1 & 2 & 1 &&&& \\
   0 & 1 & 1 &&&&&&\\
   0 &&&&&&&&
\end{tabular}\\
hence we have $P_3(1^{4};q)=4+12(1-q)+4(1-q)^2$.

Theorem 2 is a consequence of a bijection between such sequences and Grand-Dyck paths counted by double rises. Actually it is well known that $\binom{2n}{n}$ is the number of Grand-Dyck paths of semi-length $n$ and $(n+1)N(n,k)$ is the number of Grand-Dyck paths of semi-length $n$ that have $k-1$ double rises.

\end{document}